\documentclass{article}
\usepackage[latin1]{inputenc}
\usepackage{verbatim,color,amsmath}
\usepackage{amssymb}
\usepackage{amsthm}
\usepackage{mathrsfs}
\usepackage{authblk}
\usepackage{enumitem} 

\newtheorem{thm}{Theorem}[section]
\newtheorem{prop}[thm]{Proposition}
\newtheorem{lemma}[thm]{Lemma}
\newtheorem{cor}[thm]{Corollary}

\newtheorem{que}[thm]{Question}
\newtheorem{remark}[thm]{Remark}
\newtheorem{obs}[thm]{Observation}
\newtheorem{exa}[thm]{Example}

\newcommand{\stb}{, \ldots ,}

\newcommand{\N}{{\mathbb{N}}}

\newcommand{\R}{{\mathbb{R}}}

\begin{document}

\title{A characterization of $n$-associative, monotone, idempotent functions on an interval that have neutral elements}

\author{Gergely Kiss\thanks{
The first author was supported by the internal research project
R-AGR-0500-MR03 of the University of Luxembourg and the Hungarian
Scientific Research Fund (OTKA) K104178.}\\
University of Luxembourg, Faculty of Science, Mathematical Research
Unit, rue Richard Coudenhove-Kalergi 6, L-1359 Luxembourg,
gergely.kiss@uni.lu\\

\and

G\'abor Somlai \thanks{The second author was supported by the
Hungarian Scientific Research Fund (OTKA)
K115799.}\\
E\"otv\"os Lor\'and University, Faculty of Science, Institute of
Mathematics, P\'azm\'any P\'eter s\'et\'any. 1/c, Budapest, H-1117
Hungary, gsomlai@cs.elte.hu}


    \maketitle

    \begin{abstract}
    We investigate monotone idempotent $n$-ary semigroups and provide a generalization of the Czogala--Drewniak Theorem, which describes the idempotent monotone associative functions having a neutral element. We also present a complete characterization of idempotent monotone $n$-as\-so\-ciative functions on an interval that have neutral elements.


    {\bf MSR 2010 classification:} 06F05, 20M99
    \end{abstract}

    \section{Introduction}
    A function $F\colon X^n \to X$ is called $n$-\emph{associative} if for every $x_1 \stb x_{2n-1}\in X$ and every $i=1\stb n-1$, we have
    \begin{equation}
    \begin{split}
    &F(F(x_1\stb x_n),x_{n+1} \stb x_{2n-1})= \\ & F(x_1 \stb x_i
    ,F(x_{i+1} \stb x_{i+n}), x_{i+n+1}  \stb x_{2n-1}).
    \end{split}
    \end{equation}

    Throughout this paper we assume that the underlying sets of the algebraic structures under consideration are partially ordered sets (poset). Some of our results only work for totally ordered sets. In our main results we investigate $n$-ary semigroups on an arbitrary nonempty subinterval of the real numbers.

    A set $X$ endowed with an $n$-associative function $F\colon X^n \to X$ is called an $n$-\emph{ary semigroup} and is denoted by $(X,F_n)$. Clearly, we obtain a generalization of associative functions, which are the $2$-associative functions using our terminology.

    The main purpose of this paper is to describe a class of $n$-ary semigroups. An $n$-ary semigroup is called \emph{idempotent}  if  $F(a \stb a)=a$  for all $a \in X$.  Another important property is the monotonicity. An $n$-associative function is called \emph{monotone in the $i$-th variable} if for all fixed $a_1\stb a_{i-1}, a_{i+1} \stb a_n\in X$, the $1$-variable functions $f_i(x):=F(a_1 \stb a_{i-1},x,a_{i+1} \stb a_n)$ is order-preserving or order-reversing. An $n$-associative function is called \emph{monotone}
    if it is monotone in each of its variables. Further, we say that $e\in X$ is a \emph{neutral element} for an $n$-associative function $F$ if for every $x \in X$ and every $i=1\stb n$, we have $F(e \stb e, x, e \stb e) =x$, where $x$ is substituted for the $i$-th coordinate.

    An important construction of $n$-ary semigroups is the following.  Let $(X, F_2)$ be a binary semigroup. Let $F_n:=\underbrace{F_2
    \circ F_2 \circ \ldots \circ F_2}_{n-1}$, where
\begin{equation*}
  \begin{split}
    F_n(x_1 \stb x_n)&=\underbrace{F_2
    \circ F_2 \circ \ldots \circ F_2}_{n-1}(x_1 \stb x_n)\\
    &=F_2(x_1,F_2(x_2,\ldots ,F_2(x_{n-1},x_n))).
\end{split}
\end{equation*}
    We get an $n$-associative function $F_n \colon X^n \to X$ and an $n$-ary semigroup $(X,F_n)$. In this case we say that $(X,F_n)$ \emph{is derived} from the binary semigroup $(X,F_2)$ or, simply, that $F_n$ is derived from $F_2$. We also say that $(X, F_n)$ is a \emph{totally} (\emph{partially}) \emph{ordered $n$-ary semigroup} for emphasizing that $X$ is totally (partially) ordered.

    It is easy to show (see Lemma \ref{lemoda} below) that if $F_n$ is derived from $F_2$ and $F_2$ is either monotone or idempotent or has a neutral element, then so is $F_n$.

     An $n$-ary semigroup $(X,F_n)$ is called an $n$-\emph{ary group} if for each $i \in \{1 \stb n \}$, every $n-1$ elements $x_1 \stb x_{i-1}, x_{i+1} \stb x_n$ in $X$ and every $a \in X$, there exists a unique $b \in X$ with $F_n(x_1 \stb x_{i-1}, b, x_{i+1} \stb x_n)=a$. It is easy to see from the definition that ordinary groups are exactly the $2$-ary groups.

    Clearly, a function $F_n$ derived from a semigroup $F_2$ is $n$-associative but not every $n$-ary semigroup can be obtained in
    this way. Dudek and Mukhin \cite{DM} (see also Proposition \ref{propfrombinary}) proved that an $n$-ary semigroup $(X,F_n)$ is derived from a binary one if and only if $(X,F_n)$ contains a neutral element or one can adjoin a neutral element to it. As a special case of
    this theorem they obtained that an $n$-ary group is derived from a group if and only if it contains a neutral element.

    This result allows one to construct $n$-ary groups that are not derived from binary groups if $n$ is odd. Indeed, let $(X,+)$ be a group and $n=2k-1$. Define $G_n(x_1 \stb x_n):= \sum_{i=1}^{n} (-1)^{i}x_i$. It is easy to verify that $G_n$ is $n$-associative and we obtain an $n$-ary group. Moreover $G_n$ is clearly monotone. It is also easy to check that there is no neutral element for $G_n$.

    Finally, we say that an $n$-ary semigroup $(X,F_n)$  is \emph{quasitrivial} (or it is said to be \emph{conservative}) if for every $x_1 \stb x_n \in X$, we have $F_n(x_1 \stb x_n) \in \left\{ x_1 \stb x_n \right\}$. Such an $n$-variable function $F_n$  is called a \emph{choice function}. One might also say that $F_n$ preserves all subsets of $X$. Ackerman (see \cite{A}) investigated quasitrivial semigroups and also gave a characterization of them.

    Our paper is organized as follows. In Section \ref{secresults} we collect the main results proved in the paper. In Section \ref{secbinary}   we establish connections between $n$-ary semigroups and binary semigroups and prove Theorem \ref{thmbijection}. Section \ref{s3} is devoted to the proof of Theorems \ref{thm2} and  \ref{corollarymain}. Section \ref{seccr} contains a few concluding remarks.

    \section{Main results}\label{secresults}
    Let $I \subset \R$ be a  not necessarily bounded, nonempty interval. We denote by $\bar{I}$ the compact linear closure of $I$.\footnote{If $I$ is bounded and we denote the end-points of $I$ by $m$ and $M$ ($m\ge M$), then $\bar{I}=[m,M]$. If $I$ is  not bounded from below (or above), then we let $m=-\infty$ ($M=+\infty$, respectively). For instance, $\bar{\mathbb{R}}=\mathbb{R}\cup \{\pm \infty\}$.} Let  $g\colon \bar{I}\to \bar{I}$ be a decreasing function. For every $x \in I$, let $g(x-0)$ and $g(x+0)$ denote the limit of $g$ at $x$ from the left and from the right, respectively.\footnote{Let $m$ and $M$ be the boundary points of $\bar{I}$. We use the convention that $g(m-0)=M$ and $g(M+0)=m$.} We denote by $\Gamma_g$ the \emph{completed graph} of $g$, which is a subset of $\bar{I}^2$ obtained by modifying the graph of the function $g$ in the following way.
    If $x$ is a  discontinuity point of $g$, then we add a vertical line segment between the points $(x, g(x-0))$ and $(x,g(x+0))$ to extend the graph of $g$. Formally,
    $$\Gamma_g=\{(x,y)\in \bar{I}^2: g(x+0)\le y\le g(x-0)\}.$$
    We call $\Gamma_g$ (\emph{id})\emph{-symmetric} if $\Gamma_g$ is symmetric to the line $x=y$.

    The following theorem gives a description of idempotent monotone (2-ary) semigroups with neutral elements. These semigroups were first investigated by Czogala and Drewniak \cite{CD}, where the authors only dealt with closed subintervals of $\mathbb{R}$ but the statement holds for any non-empty interval. On the other hand, instead of monotonicity it was assumed that the binary function is monotone increasing. However, Lemma \ref{lemma4} shows that monotonicity implies monotone increasingness in this case.

    \begin{thm}\label{thmCD}
    Let $I$ be an arbitrary nonempty real interval. If a function $F_2\colon I^2\to I$ is associative idempotent monotone and has a
    neutral element $e \in I$, then there exits a monotone decreasing function $g\colon \bar{I}\to \bar{I}$ with $g(e)=e$ such that
    \begin{equation*}\label{eq1} F_2(x,y)=\begin{cases}
      \min{(x, y)} & \text{ if } y< g(x), \\
      \max{(x,y)}  & \text{ if } y>g(x),  \\
      \min{(x, y)} \text{ or } \max{(x,y)} & \text{ if }  y=g(x).
    \end{cases}
    \end{equation*}
    \end{thm}
    Now we present a complete characterization of idempotent monotone increasing (2-ary) semigroups with neutral elements. First
    this was proved by Martin, Mayor, and Torrens \cite{MMT} for $I=[0,1]$. Their theorem contained a small error in the description, but essentially it was correct. In the original paper \cite{MMT} the following condition for $g$ was given instead of the symmetry of $\Gamma_g$.
    The function $g\colon[0,1]\to [0,1]$ satisfies
    \begin{equation}\label{eq2}
    \inf\{y: g(y)=g(x)\}\le g^2(x)\le \sup\{y:g(y)=g(x)\} \textrm{ for
    all } x\in [0,1].
    \end{equation}
    Here (and below) $g^2(x)$ stands for $(g \circ g)(x)$.

    The authors of \cite{RTBF} proved that Theorem \ref{thm1} holds if $F_2$ is commutative also and shown that condition \eqref{eq2} is not equivalent to the symmetry of $\Gamma_g$.  Recently, Theorem \ref{thm1} was reproved in an alternative way in \cite{KLMT} for any subinterval of $\R$.

     \begin{thm}\label{thm1} Let $I\subseteq\R$ be an arbitrary nonempty interval. A function $F_2\colon I^2\to I$ is associative idempotent monotone and has a neutral element $e\in I$  if and only if there exists a decreasing function $g\colon \bar{I}\to \bar{I}$ with $g(e)=e\in I$  such that the  completed graph $\Gamma_g$ is symmetric and
    \begin{equation}\label{eq3}
    F_2(x,y)=\begin{cases}
      \min{(x, y)} & \text{ if } y< g(x) \text{ or } y=g(x) \text{ and } x<g^2(x), \\
      \max{(x,y)}  & \text{ if } y>g(x) \text{ or } y=g(x) \text{ and } x>g^2(x), \\
      \min{(x, y)}  \text{ or } \max{(x,y)} & \text{ if }  y=g(x) \text{ and } x=g^2(x).
      \end{cases}
    \end{equation}
    Moreover, $F_2(x,y)=F_2(y,x)$ except perhaps the  set of points $(x,y)\in I^2$ satisfying  $y=g(x)$ and $x=g^2(x)=g(y)$.
    \end{thm}
    If $(X,F_n)$ is an $n$-semigroup having a neutral element $e$,  then one can assign a semigroup $(X,F_2)$ to it by letting $F_2(a,b):=F_n(a,e \stb e, b)$ for every $a,b\in X$. This map $F_n\mapsto F_2$ will be denoted by $\mathcal{F}$. Our main result in Section \ref{secbinary} is the following:
    \begin{thm}\label{thmbijection}
    For any totally ordered set $X$, the map $\mathcal{F}$ is a bijection between the set of associative idempotent monotone functions on $X$ having neutral elements and the set of $n$-associative idempotent monotone functions on $X$ having neutral elements.
    \end{thm}

    We will get the following result as an easy consequence of our investigation.
    \begin{thm}\label{thm2}
    Let $I\subset\R$ be a nonempty interval, $n\ge 2$, and $F_n\colon I^n\to I$ an $n$-associative monotone increasing idempotent function with a neutral element. Then $F_n$ is quasitrivial.
    \end{thm}

    Applying Theorems \ref{thmbijection} and \ref{thm1}, we can obtain a practical method to calculate the value of $F_n(a_1 \stb a_n)$ for any $a_1\stb a_n\in I$, where $I \subset \R$ is an interval.

    For every decreasing function $g\colon \bar{I} \to \bar{I}$ a pair $(a,b)\in I^2$ is called \emph{critical} if $g(a)=b$ and $g(b)=a$. By Theorem \ref{thm1} and Lemma \ref{lemma4}, for every idempotent monotone semigroup $(X,F_2)$ with a neutral element, there exists a
    unique decreasing function $g$ satisfying \eqref{eq3}. Theorem \ref{thm1} shows also that $F_2$ commutes on every non-critical pair $(x,y)\in I^2$ (i.e., $F_2(x,y)=F_2(y,x)$). Since for a critical pair $(a,b)$ the value of $F_2(a,b)$ and $F_2(b,a)$
    can be independently chosen from $g$, we have two cases. We might have that $F_2$ commutes on $a,b$  or not. A pair $(a,b)$ is called \emph{extra-critical} if $F_2(a,b) \ne F_2(b,a)$. We note that being critical or extra-critical are both symmetric relations.

    Finally, in order to simplify notation and give a compact way to express the value of $F_n$ at some $n$-tuple $(a_1 \stb a_n)$ of the elements from a totally ordered set, we introduce the following.  The smallest and the largest elements of the set $\{a_1 \stb a_n \}$ are denoted by $c$ and $d$, respectively. There exist $i,j$ with $1 \le i \le j \le n$ such that $a_i=c \mbox{ or } d$, $a_j=c \mbox{ or } d$ and $a_k \ne c \mbox{ and } d$ for every $k <i$ and $k>j$. We write $e_1:=a_i$ and $e_2:=a_j$.

    The following statement was proved in \cite{DM}:
    \begin{prop}[Dudek, Mukhin]\label{propfrombinary}
    If $(X,F_n)$ is an $n$-ary semigroup with a  neutral element $e$, then $F_n$ is derived from a binary function $F_2$, where \begin{equation}\label{eq2n}
        F_2(a,b):=F_n(a,e\stb e,b).
    \end{equation}
\end{prop}
    \begin{thm}\label{corollarymain}
    Let $F_n\colon I^n\to I$ be an $n$-associative idempotent function with a neutral element that is monotone in its first and last coordinates.     If $(c,d)$ is a not an extra-critical pair, then $F_n(a_1 \stb a_n)=F_2(c,d)$.  If $(c,d)$ is an extra-critical pair, then $F_n(a_1 \stb a_n)=F_2(e_1,e_2)$.
    \end{thm}
    Now we point out three important consequences of Theorem \ref{corollarymain}. First we generalize Czogala--Drewniak's theorem (Theorem \ref{thmCD}) as follows.
    \begin{thm}\label{thmgen1}
    Let $I$ be an arbitrary nonempty real interval. If a function $F_n\colon I^n\to I$ is $n$-associative idempotent monotone and has a
    neutral element $e \in I$, then there exits a monotone decreasing function $g\colon \bar{I}\to \bar{I}$ with $g(e)=e$ such that $\Gamma_g$ is symmetric and
    \begin{equation*}\label{eqalt} F_n(a_1\stb a_n)=\begin{cases}
      c & \text{ if } c < g(d), \\
      d  & \text{ if } c > g(d),  \\
      c \text{ or } d & \text{ if }  c=g(d),
    \end{cases}
      \end{equation*}
     where $c$ and $d$ denote the minimum and the maximum of the set $\{a_1\stb a_n\}$, respectively.
    \end{thm}
    We note that a generalization of Theorem \ref{thm1} is essentially stated in Theorem \ref{corollarymain}. In \cite{RTBF} the authors investigated idempotent uninorms, which are associative, commutative, monotone functions with a  neutral element and idempotent also. We introduce \emph{$n$-uninorms}, which are $n$-associative, commutative, monotone functions with neutral element.
    Here we show a generalization of \cite[Theorem 3]{RTBF} for $n$-ary operations.
    \begin{thm}
    An $n$-ary operator $U_n$ is an idempotent $n$-uninorm on $[0, 1]$ with a neutral element $e \in [0,1]$ if and only if there exists a decreasing function $g \colon [0, 1] \to [0, 1]$ with $g(e)=e$ and with symmetric graph $\Gamma_g$ such that
    \begin{equation}\label{eq123} U_n(a_1\stb a_n)=
    \begin{cases}
    c&\textrm{ if } c < g(d)) \textrm{ or } d < g(c),\\
    d&\textrm{ if } c > g(d)\textrm{  or } d > g(c), \\
    c\textrm{ or } d  &\textrm{ if } c=g(d) \textrm{ and } d = g(c),
    \end{cases}
    \end{equation}
    where $c$ and $d$ are as in Theorem \ref{thmgen1}. Moreover, if $(c,d)$ is a critical pair $(c=g(d), d = g(c))$, then the value of $U_n(a_1\stb a_n)$ can be chosen to be $c$ or $d$ arbitrarily and independently from other critical pairs.
    \end{thm}

    We may generalize our concept in the following way. Let $X^*=\bigcup_{n\in \N} X^n$ be the set of finite length words over the alphabet $X$. A multivariate function $F\colon X^* \to X$ is \emph{associative} if it satisfies $$F({\bf x}, {\bf x'})=  F(F({\bf x}), F({\bf x'}))$$ for all ${\bf x}, {\bf x'}\in X^*$. It is easy to check that $F|_{X^n}$ is $n$-associative for every $n\in \mathbb{N}$. We say that $F$ is idempotent or monotone or that it has a neutral element if so are the functions $F|_{X^n}$ for every $n\in\N$.

    \begin{thm}
    Let $I$ be a nonempty real interval. Then $F\colon I^*\to I$ is associative idempotent monotone and has a neutral element if and only if there is a decreasing function $g\colon \bar{I}\to \bar{I}$ with symmetric completed graph $\Gamma_g$ such that $F|_{X^2}$ satisfies \eqref{eq3}. Furthermore $F$ must be monotone increasing in each variable.
    \end{thm}
    Concerning to associativity of multivariate functions the interested reader is referred to \cite{GMMP,EMT}.

   \section{From $n$-ary to binary semigroups}\label{secbinary}
    In this section we prove Theorem \ref{thmbijection}. Therefore the main purpose of this section is to transfer properties from an $n$-ary semigroup to the corresponding binary semigroup. We start with the converse. We have already mentioned that, given a semigroup $(X,F_2)$, one can easily construct the $n$-ary semigroup $(X,F_n)$, where $F_n=\underbrace{F_2 \circ \ldots \circ F_2}_{n-1}$.
    The following lemma is an easy consequence of the definitions.
    \begin{lemma}\label{lemoda}
    Let $(X,F_2)$ be a partially ordered semigroup. If $F_2$ has any of the following properties
    \begin{enumerate}[label=(\roman*)]
        \item monotone
        \item\label{lemodab} idempotent
        \item has a neutral element
    \end{enumerate}
    then so does the function $F_n$.
    \end{lemma}
      \begin{obs}\label{obsneutral}
    If $F_2$ is defined by \eqref{eq2n}, the element $e$ is also a neutral element of  $F_2$ since $F_2(e,a)=F_n(e\stb e,a)=a= F_n(a,e \stb e)=F_2(a,e)$ for every $a\in X$.
    \end{obs}

    \begin{lemma}\label{lemma1}
    Let $F_n\colon X^n\to X$ be an $n$-associative function on the partially ordered set $X$. Assume $F_n$ is idempotent and monotone in
    the first and the last coordinates and is derived from an associative function $F_2$. Then $F_2$ is monotone.
    \end{lemma}
    \begin{proof}
    We show that if $F_n$ is monotone in its last coordinate then so is $F_2$. Take an arbitrary $a \in X$ and let $b=F_{n-1}(a \stb a)$. In this case $F_2(b,a)=F_n(a \stb a)=a$. Substituting $a=F_2(b,a)$, we obtain $a=F_2(b,F_2(b,a))$. Using the same substitution $n-2$ times, we get that $a$ can be expressed as $F_{n-1}(c_1 \stb c_{n-1})$ for some $c_1 \stb c_{n-1}$. Then $F_2(a,x)=F_n(c_1 \stb c_{n-1},x)$ is clearly monotone in its last coordinate.

    Similarly, $F_2$ is monotone in its first coordinate if $F_n$ is.
    \end{proof}

    \begin{remark}\label{remark1}
    {\rm If $F_n$ is $n$-associative idempotent and monotone in the first and the last variables on a poset $X$, then, by Lemma \ref{lemma1}, $F_2$ is also monotone. It is easy to show that $F_k:=\underbrace{F_2\circ \dots \circ F_2}_{k-1}$ is $k$-associative and  monotone in each variable. In particular, $F_n$ is monotone in each of its variables.}
    \end{remark}

    \begin{lemma}\label{folemma}
    Let $F_n\colon X^n\to X$ be an $n$-associative function on a \textbf{totally ordered} set. Assume $F_n$ is idempotent and monotone in each variable and $F_n$ has a neutral element or is derived from an associative function $F_2$. Then $F_2$ is idempotent as well.
    \end{lemma}

    \begin{proof}
    We prove that $F_k=F_2 \circ \ldots \circ F_2$ is idempotent for every $2 \le k \le n$. We use backward induction. Arguing by contradiction,  assume that for some $k$ with $3\le k\le n$ there exists $a \in X$ such that has $F_{k-1}(a\stb a)=b\ne a$ and by the inductive hypothesis $F_k(x \stb x)=x$ for every $x\in X$. We note that the second condition holds for $k=n$, since $F_n$ is idempotent. Clearly, we may assume without any loss that $a<b$. We compare the following terms:
    \begin{table}[h]
    \caption{ \label{table2}}
    \bigskip
    \begin{tabular}{l|l|l|l|l}
    $F_k(a\stb  a,b)$ & $F_k(a \stb a, b,b)$ & $F_k(a\stb a,b,b,b)$&\dots& $F_k(a,b \stb b,b)$\\
    ~&~&~&~&~\\
    $F_k(a \stb  a,a)$ & $F_k(a\stb a, b,a)$ & $F_k(a\stb a,b,b,a)$&\dots&
    $F_k(a,b \stb b,a)$
    \end{tabular}
    \end{table}

    The function $F_k$ is monotone in each variable by Remark \ref{remark1}. Observe  that in Table \ref{table2} the elements in each column only differ in the last coordinate. Hence each of the elements in the lower row is not greater than the element above it by the monotonicity of $F_k$.

    Now we calculate expressions in Table \ref{table2}. It is clear that $F_k(a \stb a)=a$ by the inductive assumption.
    Before we continue, we present two useful lemmas.
    \begin{lemma}\label{lemma2}
    Let $a$ and $b$ be as above. Further, let $x_1= \ldots =x_l=a$ and  $x_{l+1}= \ldots =x_k=b$. Then for every $\pi \in Sym(k)$ we have
    \[
    F_k(x_1  \stb x_k) =F_k(x_{\pi(1)} \stb x_{\pi(k)}).
    \]
    \end{lemma}

    \begin{proof}
    Substituting $b=F_{k-1}(a\stb a)$ in the expression above, it is easy to see that we may rearrange $a$'s and $b$'s arbitrarily.
    \end{proof}

    \begin{lemma}\label{lemma3} Let $l$ and $m$ be fixed and $l+m=k$. If $1\le m\le k-2$, then $$F_k(\underbrace{a\stb a}_l,\underbrace{b \stb b}_m)= F_{l}(\underbrace{a\stb a}_{l}).$$
    In particular, if $m=k-1$, then
    $$F_k(a,\underbrace{b \stb b}_{k-1})=a.$$
    \end{lemma}
    \begin{proof}
    A direct calculation shows that the statement holds. Indeed,
    \[
    F_k(\underbrace{a\stb a}_l,\underbrace{b \stb b}_m)=F_k ( \underbrace{a\stb a}_l,\underbrace{F_{k-1}(a\stb a) \stb F_{k-1}(a\stb a)}_m).\]
    Now using associativity of $F_2$ and idempotency of $F_k$,
    we obtain
    \[
    F_2(a,F_{k-1}(a  \stb a))=F_k(a  \stb a) =a.
    \]
    Applying Lemma \ref{lemma2} and the previous observation $m$ times,
    we obtain \[ F_k ( \underbrace{a\stb a}_l,\underbrace{F_{k-1}(a\stb
    a) \stb F_{k-1}(a\stb a)}_m)=F_l(a \stb a).\]
    \end{proof}

    Using Lemma \ref{lemma2} and Lemma \ref{lemma3}, we get that
    \begin{equation}\label{eqsys}
    \begin{split}
    &F_k(a\stb a)=a\\
    F_k(a\stb a,  a, b,a)=&F_k(a\stb a, a,a,b)= F_{k-1}(a\stb a),\\
    F_k(a\stb a, b, b,a)=&F_k(a\stb a, a,b,b)= F_{k-2}(a\stb a),\\
    \vdots &~~~~~~~~~~~~~~~~~~~ \vdots\\
    F_k(a,b, b\stb b,a)=&F_k(a,a,b\stb b,b)=F_{2}(a,a),\\
    &F_k(a,b\stb b,b)=a.
    \end{split}
    \end{equation}
    Since in each column of the table we just change the last
    coordinate, we can use monotonicity. We note that $F_k$ is increasing (order-preserving) in the last variable since $F_k(a\dots,a, b)=b>a=F_k(a,\dots,a,a)$. Substituting the results of \eqref{eqsys} into the table, we get the following:
    \[
    \begin{array}{l|l|l|l|l}
     F_{k-1}(a\stb a) & F_{k-2}(a\stb a) & F_{k-3}(a\stb a)&\dots& a\\
    \uparrow&\uparrow&\uparrow&\uparrow&\uparrow\\
    a & F_{k-1}(a\stb a)& F_{k-2}(a\stb a)&\dots& F_2(a,a)
    \end{array}\]
    Here the notation $\uparrow$ means that an element in the lower row is less than or equal to
    the corresponding element in the upper row. Thus,
    $$a\le F_{k-1}(a\stb a)\le F_{k-2}(a\stb a)\le \dots \le  F_2(a,a)\le a.$$
    This gives
    $$a= F_{k-1}(a\stb a)= F_{k-2}(a\stb a)= \dots = F_2(a,a),$$
    a contradiction, since $F_{k-1}(a,\stb a)=b\ne a$ by our assumption. This shows that $F_k$ is idempotent for every $k\ge2$, finishing the proof of Lemma \ref{folemma}.
    \end{proof}
    The underlying set of the $n$-associative function in Lemma \ref{folemma} is totally ordered. The following example shows that this requirement is essential.
\begin{exa}
\rm{For $k \ge 3$ we construct a $k$-ary semigroup $(X,F_k)$, which
is derived from a non-idempotent semigroup $(X,F_2)$, where $F_2$ is
monotone in both of its variables and has a neutral element. Thus,
$F_{k-1}$ and $F_k$ are also monotone having neutral element. We
show that $F_k$ is idempotent and $F_{k-1}$ is not idempotent, thus
$F_2$ cannot be idempotent by Lemma \ref{lemoda} \ref{lemodab}. This
example shows that the condition that $X$ is a totally ordered set
is crucial in Lemma \ref{folemma}.

Let $X=\{m,M \} \cup Z_{k-1}$, where $Z_{k-1}$ is the cyclic group
of order $k-1$. We define a partial ordering on $X$ in the following
way. $M$ and $m$ are the largest and smallest elements of $X$,
respectively. The elements of $Z_{k-1}$ are mutually incomparable
but they are all larger than $m$ and smaller than $M$. The set $X$
endowed with this partial ordering is a modular lattice. Further we
build an associative function $F_2$ as follows:
\[ F_2(x,y)=
\begin{cases}
  M & \text{if } x=M \mbox{ or } y=M, \\
  m & \text{if } x=m \mbox{ or } y=m \mbox{ and } x,y <M, \\
  xy & \text{if } x,y \in Z_{k-1}.
\end{cases}
\]
It is easy to verify that $F_2$ is associative and monotone
increasing in both of its variables. The identity element $e$ of
$Z_{k-1}$ is the neutral element of $(X,F_2)$. One can define
$F_{k-1}$ and $F_{k}$ as before. By Lemma \ref{lemoda} the functions
$F_{k-1}$ and $F_k$ are $(k-1)$- and $k$-associative functions,
respectively. Both of them are monotone having neutral element.
Finally, it is easy to check that $F_{k-1}$ is not idempotent since
$F_{k-1}(a \stb a)=e$ for every $a \in Z_{k-1}$ while $F_{k}(x \stb
x)=x$ for every $x \in X$. Note that the cyclic group $Z_{k-1}$
might have been substituted by any nontrivial group whose exponent
divides $k-1$.}
\end{exa}
\begin{remark}
  {\rm  For distributive lattices the statement of Lemma \ref{folemma} seems true, but a potential proof would be basically different from the proof of the lemma. Thus, it goes beyond the topic of the current paper. (See also Question \ref{q2}
  in Section \ref{seccr}.)}
\end{remark}
    The following lemma provides extra information about monotone, associative and idempotent semigroups.
 \begin{lemma}\label{lemma4}
    Let $X$ be a partially ordered set. If $F_2\colon X^2\to X$ is associative idempotent and monotone in each variable, then $F_2$ is monotone increasing in each variable.
    \end{lemma}
    \begin{proof} Assume that $F_2$ is not monotone increasing in each variable. Let us assume that $F_2$ is decreasing in the second variable. We also exclude the case when $F_2$ is both increasing and decreasing in the second variable (i.e., $F_2(x, \cdot)$ is constant for any fixed $x\in X$), so that we may assume that there exist $x,y,z \in X$ such that $y <z$ and $F_2(x,y) >F_2(x,z)$.

    Now by the idempotency of $F_2$ we have $F_2(F_2(x,x),y)=F_2(x,y)$ and $F_2(F_2(x,x),z)=F_2(x,z)$. Our assumption
    then gives
    $$F_2(F_2(x,x),y)>F_2(F_2(x,x),z).$$
Using the associativity of $F_2$ we get $F_2(x,F_2(x,y)) >
F_2(x,F_2(x,z))$.

On the other hand, since $F_2(x,y) >F_2(x,z)$ and $F_2$ is
decreasing in the second variable we get $F_2(x,F_2(x,y))  \le
F_2(x,F_2(x,z))$, which contradicts our assumption.

One can get the same type of contradiction if we switch the role of
the coordinates. Thus, $F_2$ is monotone increasing in both
    variables.
    \end{proof}

    \begin{remark}
    {\rm The following examples demonstrate that if we omit any of the conditions of Lemma \ref{lemma4}, the conclusion of the lemma fails.
    \begin{enumerate}
    \item  Let $F_2(x,x)=x$ for $x\in \R$ and $F_2(x,y)=0$ if $x,y\in \R$, $x \ne y$. Then $F_2$ is associative and idempotent, but not monotone in each variable.
    \item Let $F_2(x,y)=2x-y$ for $x,y\in \R$. Then $F_2$ is idempotent and monotone in each variable, but not associative and clearly not monotone increasing.
    \item Let $F_2(x,y)=-x$, if $x,y>0$, and $F_2(x,y)=0$ otherwise. Then $F_2$ is associative, since $F_2(x, F_2(y,z))=F_2(F_2(x,y),z))=0$ and $F_2$ is monotone decreasing in each variable and $F_2$ is not idempotent.
    \end{enumerate}}
    \end{remark}

    \begin{cor}\label{cor1}
    If $(X, F_n)$ is a totally ordered $n$-ary semigroup, where $F_n=\underbrace{F_2\circ F_2\circ \dots\circ F_2}_{n-1}$ is
    idempotent and monotone in the first and the last variables, then $F_n$ is monotone {\rm increasing} in {\rm each} variable.
    Moreover, $F_k=\underbrace{F_2\circ \dots\circ F_2}_{k-1}$ is {\rm monotone increasing} for every $k\ge2$.
    \end{cor}
    \begin{proof}
    By definition, $F_2$ is associative. Since $F_n$ is monotone in each variable, so is $F_2$ by Lemma \ref{lemma1}. By Lemma \ref{folemma}, $F_2$ is idempotent. Thus by Lemma  \ref{lemma4}, it is monotone increasing. Thus
    $F_k=F_2\circ F_2\circ \dots\circ F_2$ is also  monotone increasing for every $k\ge2$.
    \end{proof}

    If $F_n$ is $n$-associative and has a neutral element, then there exists $F_2$ such that $F_n=F_2\circ F_2\circ \dots\circ F_2$.
    Using the results of this section, we prove the following proposition.
    \begin{prop}\label{prop}
    Let $(X,F_n)$ be a totally ordered $n$-ary semigroup, which is monotone idempotent and has a neutral element. Then $F_n$ is derived from a binary semigroup  $(X,F_2)$, where $F_2$ is also monotone idempotent and it also has a neutral element. Moreover $F_n$ is monotone increasing in each variables.
    \end{prop}
    \begin{proof}
    Since $F_n$ is idempotent $n$-associative and has a  neutral element, it follows from Proposition \ref{propfrombinary} that $F_n=F_2\circ \dots  \circ F_2$, where $F_2\colon X^2\to X$ is associative. By Lemmas \ref{lemma1}, \ref{folemma}, and \ref{lemma4}, $F_2$ is monotone increasing and  idempotent.  By Observation \ref{obsneutral} that in this case $F_2$ has a neutral element, as well.
    \end{proof}

    \emph{Proof of Theorem \ref{thmbijection}.}
    By Proposition \ref{prop}, every $n$-associative function $F_n$ which is monotone idempotent and has a neutral element, is derived from an associative function $F_2$ determined by $F_2(a,b):=F_n(a,e \stb e,b)$ which is monotone idempotent and has a neutral element.  Then we have
    \begin{equation}\label{eqn2}
    F_n(a\stb a,b)=F_n(a,e\stb e, b)=F_n(a,b\stb b)=F_2(a,b).
    \end{equation}
    Recall that the map that assigns $F_2$ to $F_n$ was denoted by $\mathcal{F}$. By Lemma \ref{lemoda} and Proposition \ref{prop}, for every $F_2$, there exists $F_n$ satisfying \eqref{eqn2} whence $\mathcal{F}$ is surjective. The map $\mathcal{F}$ is injective since $F_2$ uniquely defines $F_n$.  This finishes the proof of Theorem \ref{thmbijection}.\qed

    \begin{remark}
    {\rm Using Corollary \ref{cor1} we may weaken the assumptions of Theorem \ref{thmbijection}, where $F_n$ is assumed to be monotone in each variable. Instead, we might have assumed that $F_n$ is monotone in the first and last variables.}

    \end{remark}
    \begin{lemma}\label{lemma8.1}
    Let $(X,F_n)$ be a totally ordered $n$-ary semigroup derived from $(X,F_2)$, where $F_2$ is idempotent, associative, monotone increasing and has a neutral element. Then
    \begin{equation}\label{eq8.1}
    F_n(a,y_1,\dots, y_{n-2},b)=F_2(a,b)
    \end{equation}
    whenever $a\le y_1\stb  y_{n-2}\le b.$
    \end{lemma}
    \begin{proof}
    By Theorem \ref{thmbijection}, the function $F_n$ is monotone, and therefore, the claim directly comes from \eqref{eqn2}.
    \end{proof}

\section{Proof of the main results}\label{s3}
\emph{Proof of Theorem \ref{thm2}:}
    It follows from Proposition \ref{prop} that $F_n$ is derived from an associative function $F_2$. Moreover $F_2$ is monotone, idempotent and has a neutral element. Therefore, we may apply Theorem \ref{thm1} in a special form  which we obtain when $F_2$ is a choice function (i.e., when $(X,F_2)$ is quasitrivial). Since $F_n$ is obtained as the composition of $n-1$ copies of $F_2$, we get that $F_n$ is also a choice function.
    \qed

    \smallskip

    \emph{Proof of Theorem \ref{corollarymain}.}
    First assume that $c$ and $d$ commute with every element of the set $A:=\{a_1\stb a_n\}\subset I$. In this case we may assume, using  idempotency of $F_2$,  that there exists $k \le n$ such that $F_n(a_1 \stb a_n)=F_k(c,a'_2 \stb a'_{k-1},d)$ and $c < a'_i < d$ for all $i=2 \stb k-1$. By Proposition \ref{prop}  we can apply Lemma \ref{lemma8.1} that gives $F_n(a_1 \stb a_n)=F_2(c,d)$.

    Now assume that $d$ does not commute with an element of $A$ but $c$ commutes with all of them. In this case $g(d) \in A$ is the one not commuting with $d$. Since $c$ is the smallest element of $A$ we get $c<g(d)$. Further, $d$ is the largest element of $A$ and $g$ is decreasing so $g(a_i)>c$ for all $i =1 \stb n$. Theorem \ref{thm1} gives $F_2(c,a_i)=F_2(a_i,c)=c$ for all $i=1 \stb n$. Therefore $F_n(a_1 \stb a_n)=c$. Since $F_2(c,a_i)=c$ for every $i$, we get $F_n(a_1 \stb a_n)=F_2(c,d)=c$.
    A similar argument shows that $F_n(a_1 \stb a_n)=d=F_2(c,d)$ if $c$ and $d$ switch the roles.

    Finally, assume that neither $c$ nor $d$ commutes with every element of $A$. In this case the set $A$ contains $g(c)$ and $g(d)$ and $g(g(c))=c$, $g(g(d))=d$. We claim that  $g(c)=d$ and $g(d)=c$. Indeed, if $g(c) \in A$, then $g(c)\le d$ since $d$ is the largest element of $A$, and similarly $g(d)\ge c$. Since $g$ is monotone and $g(g(d))=d$, we get $d=g(g(d))\le g(c)$. Therefore $g(c)=d$. Similarly using $c=g(g(c))\ge g(d)$ we get $g(d)=c$. What we obtained is that $(c,d)$ is an extra-critical pair in this case.

    Now $c$ and $d$ are the elements in $A$ that do not commute. This also implies that $c$ and $d$ commute with all other element of $A$.

    By definition of $e_1$ and $e_2$ (see Section \ref{secresults}), $e_1$ and $e_2$ are the value of the first and respectively the last appearance of $c$ or $d$.  Since $e_1$ commutes with its left neighbours  and $e_2$ commutes with its right neighbours, we may assume that $a_1=e_1$ and $a_2=e_2$.
    We get the following cases:
    \begin{enumerate}
        \item[(i)]If $e_1\ne e_2$, then by Lemma \ref{lemma8.1} $F_n(e_1\stb e_2)=F_2(e_1, e_2).$
        \item[(ii)] If $e_1= e_2$, then we show that $F_n(e_1\stb e_2)=F_2(e_1,e_2)=e_1.$
    \end{enumerate}
    Using Lemma \ref{lemma8.1} for arbitrary number of variables, we get that every subsequence of $a_1\stb a_n$ consisting of elements lying strictly between $c$ and $d$ can be eliminated. Thus, one can write $F_n(a_1 \stb a_n)=F_k(b_1\stb b_k)$,
    where $k \le n$ and $b_i=c \mbox{ or } d$ for every $1 \le i \le k$ and, in our case, $b_1=b_k=e_1$. Since $F_2$ is idempotent, we may assume $b_i \ne b_{i+1}$ for $1 \le i \le k-1$. Using idempotency and associativity of $F_2$ again, we have $F_2(F_2(c,d),F_2(c,d))=F_2(c,d)$ and $F_2(F_2(d,c),F_2(d,c))=F_2(d,c)$. Therefore $F_k(b_1 \stb b_k)$ can be reduced to either
    $F_3(c,d,c)$ or $F_3(d,c,d)$.

    If $F_2(c,d)=c$, then $F_3(c,d,c)=F_2(F_2(c,d),c)=F_2(c,c)=c =e_1$.

    If $F_2(c,d)=d$, then $F_3(c,d,c)=F_2(F_2(c,d),c)=F_2(d,c)$. Since $c$ and $d$ do not commute  we get $F_2(d,c)=c=e_1$.

    Similarly, one can verify that $F_3(d,c,d)=d=e_1$.  This finishes the proof of Theorem \ref{corollarymain}.
    \qed

    \section{Concluding remarks}\label{seccr}
    In this paper we have investigated the $n$-ary associative, idempotent, monotone functions $F_n\colon X^n\to X$ that have neutral elements. We have shown that such an $F_n$ in general setting when the underlying set $X$ is totally ordered implies the existence of binary functions $F_2\colon X^2\to X$ with similar properties such that $F_n$ is derived from $F_2$. However many of the properties of $F_n$ are inherited by $F_2$ if $X$ is only a partially ordered set. We summarize the results of Section \ref{secbinary} (if $X$ is a totally ordered set) in the following table.
    \begin{center}
    \begin{tabular}{l l l}
    Properties of  $F_n$  &  & Properties of  $F_2$\\
    \hline
    $n$-\textrm{assoc. with a neutral element}& $\Longrightarrow$& \textrm{assoc. with a neutral element}\\
    \hline
    \textrm{Now we assume } $F_n=F_2\circ\dots \circ F_2$:&&\\
    \hline
    $n$-\textrm{assoc., idempotent, monotone}&$\Longrightarrow$& \textrm{monotone}\\
    $n$-\textrm{assoc., idempotent, monotone}&$\Longrightarrow$& \textrm{idempotent}\\
    $n$-\textrm{assoc., idempotent, monotone}&$\Longrightarrow$& \textrm{monotone increasing}\\
    \hline
    \textrm{Some easy observations show:}&&\\
    \hline
    $n$-\textrm{associative}&$\Longleftarrow$& \textrm{associative}\\
    \textrm{monotone increasing} &$\Longleftarrow$& \textrm{monotone increasing}\\
    \textrm{idempotent} &$\Longleftarrow$& \textrm{idempotent}\\
    \textrm{has a neutral element} &$\Longleftarrow$& \textrm{has a neutral element}\\
    \hline
    \textrm{Thus:} &&\\
    \hline
    $n$-\textrm{assoc., idempotent, mon. incr.}&$\Longleftrightarrow$& \textrm{assoc., idemp., mon. incr.}\\
    \textrm{with a neutral element }&~& \textrm{with a neutral element}

    \end{tabular}
    \end{center}

    In the main results we have obtained a characterization of $n$-associative, idempotent, monotone functions on any (not necessarily bounded) subinterval of $\R$ in the spirit of the characterization of the binary case. We also generalize the classical Czogala--Drewniak theorem. In addition, we get that every $n$-associative, idempotent and monotone function with a neutral element must be quasitrivial (conservative).

    Further improvement would be based on the elimination of any of the properties of $F_n$. The most crucial property seems to be that $F_n$ has a neutral element since all of our results based on this condition as otherwise $F_n$ is not necessarily derived from $F_2$. On the other hand, in \cite{MMT} one can be found a characterization of associative, conservative, monotone increasing, idempotent binary functions defined on $[0,1]$ without assumption of having a neutral element. Therefore, we suggest the following questions for further investigation.
    \begin{que}
    How can we characterize $n$-associative, monotone, idempotent functions on a subinterval of $\R$?
    \end{que}

    \begin{que}\label{q2}
    How can we characterize $n$-associative functions on distributive lattices provided that the functions are monotone idempotent and have neutral elements?
    \end{que}


    \end{document}